\documentclass{amsart}
\usepackage{amsmath}
\usepackage{graphicx} 
\usepackage{amsfonts} 
\usepackage{amssymb}
\usepackage{pdfsync}
\usepackage{color}
\usepackage{mathrsfs}
\usepackage{epsfig}
\usepackage{subfig}
\usepackage{url}
\theoremstyle{plain}
\newtheorem{theorem}{Theorem}[section]
\newtheorem{example}{Example}[section]
\newtheorem{lemma}[theorem]{Lemma}
\newtheorem{proposition}[theorem]{Proposition}
\newtheorem{corollary}[theorem]{Corollary}
\newtheorem{remark}[theorem]{Remark}
\newtheorem{definition}[theorem]{Definition}
\theoremstyle{definition}
\theoremstyle{remark}
\numberwithin{equation}{section}

\newcommand{\co}{\mathrm{co}\,}
\newcommand{\pl}{\mathrm{pl}}
\newcommand{\bi}{\bar\imath}
\newcommand{\bj}{\bar\jmath}

\newcommand{\bT}{T}

\newcommand{\T}{\mathcal{T}}

\newcommand{\ep}{\varepsilon}
\newcommand{\e}{\varepsilon}

\newcommand{\dd}{\displaystyle}

\newcommand{\R}{\mathbb{R}}
\newcommand{\N}{\mathbb{N}}
\newcommand{\Z}{\mathbb{Z}}

\newcommand{\F}{\mathcal{F}}

\newcommand{\di}{\textrm{dist}}

\newcommand{\ud}{\;\mathrm{d}}

\newcommand{\Om}{\Omega}
\newcommand{\supp}{\mathrm{supp}\,}

\newcommand{\AF}{\mathcal{AF}}

\newcommand{\weakly}{\rightharpoonup}           
\newcommand{\weakstar}{\stackrel{*}{\weakly}}   
\newcommand{\fla}{\stackrel{\mathrm{flat}}{\rightarrow}}
\newcommand{\flt}{\mathrm{flat}}
\newcommand{\W}{\mathbb{W}}

\newcommand{\loc}{\mathrm{loc}}

\def\disstilde{ \tilde D_\phi}
\def\diss{ D_\phi}

\newcommand{\res}{\mathop{\hbox{\vrule height 7pt width .5pt depth 0pt
\vrule height .5pt width 6pt depth 0pt}}\nolimits}
\usepackage{latexsym}
 
\title[Motion of discrete screw dislocations along glide directions] {Minimising movements for the motion of discrete screw dislocations along glide directions}

\author[R. Alicandro]
{R. Alicandro}
\address[Roberto Alicandro]{DIEI, Universit\`a di Cassino e del Lazio meridionale, via Di Biasio 43, 03043 Cassino (FR), Italy}
\email[R. Alicandro]{alicandr@unicas.it}
\author[L. De Luca]
{L. De Luca}
\address[Lucia De Luca]{Zentrum Mathematik - M7, Technische Universit\"at M\"unchen, Boltzmannstrasse 3, 85748 Garching, Germany}
\email[L. De Luca]{deluca@ma.tum.de}
\author[A. Garroni]
{A. Garroni}
\address[Adriana Garroni]{Dipartimento di Matematica ``Guido Castelnuovo'', Sapienza Universit\`a di Roma, P.le Aldo Moro 5, I-00185 Roma, Italy}
\email[A. Garroni]{garroni@mat.uniroma1.it}
\author[M. Ponsiglione]
{M. Ponsiglione}
\address[Marcello Ponsiglione]{Dipartimento di Matematica ``Guido Castelnuovo'', Sapienza Universit\`a di Roma, P.le Aldo Moro 5, I-00185 Roma, Italy}
\email[M. Ponsiglione]{ponsigli@mat.uniroma1.it}

\begin{document}

\maketitle

\begin{abstract}

In \cite{ADGP2} a  simple discrete scheme for the motion of screw dislocations toward low energy configurations has been proposed. 
There, a formal limit of such a scheme, as the  lattice spacing and the time step tend to zero, has been described. The limiting  dynamics agrees with the maximal dissipation criterion introduced in \cite{CG} and  predicts motion along the glide directions of the crystal. 

In this paper, we provide rigorous proofs of  the results in \cite{ADGP2}, and in particular of the passage from the discrete  to the continuous dynamics. The proofs are based on $\Gamma$-convergence techniques.  
\end{abstract}

\section*{Introduction}\label{intro}
This paper deals with variational models describing the motion of straight screw dislocations toward low energy configurations \cite{ADGP1, ADGP2, BFLM, CF,  KW}.  
Here we provide  rigorous justifications to the results announced, and described in a more mechanical language in the companion paper \cite{ADGP2}.

In a previous paper \cite{ADGP1}, we have considered  a discrete anti-plane model for elasticity in a cubic lattice, governed by periodic nearest neighbors interactions. In view of the anti-plane assumption, all the relevant quantities are defined in a cross section  of the crystal, i.e., on a square lattice. 
Following the formalism in \cite{AO}, we have introduced screw dislocations in the model, as point topological singularities of the discrete displacement field.
First, we have analysed  by means of a  $\Gamma$-convergence expansion the elastic energy induced by dislocations, as the lattice spacing $\e$ tends to zero, showing that the energy concentrates on points  which interact through the so-called renormalised energy.
Then, 
we have proposed and analysed an implicit Euler scheme for the gradient flow of the energy induced by discrete dislocations. This consists in introducing a time step $\tau$ and considering a step by step minimisation of a total energy, given by the sum of the elastic energy induced by a distribution of dislocations, and a dissipation energy spent to move the singularities. 
We have proved that
the minimising movements scheme allows dislocations to overcome the energy barriers due to  the discrete structure. Indeed, in the limit as first $\e$ and then $\tau$ tend to zero, the dynamics is driven by the gradient flow of the renormalised energy (in perfect analogy with the theory of dynamics of vortices in superconductivity \cite{BBH, SS2}). 
  
In \cite{ADGP2} we have  extended the discrete model studied in \cite{ADGP1}  to several crystal structures relevant for applications and we have  enriched the scheme for the motion of screw dislocations  by considering  new anisotropic  rate dependent dissipations, related to the specific crystalline structures.
We have shown, in the formal limit as $\e,\tau\to 0$, that the proposed scheme is able to predict motion of dislocations along the glide directions of the crystal, according with the {\it maximal energy dissipation criterion} postulated in \cite{CG}. 
The latter asserts that the dislocations move along the glide directions that maximise the scalar product with the Peach K\"ohler force $j$ which, in turn, is given by  $-\nabla W$, being $W$ the renormalised energy. 
According to this criterion the velocity field might be not uniquely determined and the formulation needs to be relaxed. The effective dynamics is then  described by a differential inclusion rather than a differential equation (as analysed in \cite{BFLM}).

In this paper, starting from a microscopic fundamental description, we give a rigorous derivation of this effective dynamics, highlighting  its gradient flow structure and providing an approximation scheme and a selection principle for the dynamics proposed in \cite{CG}.
Our results are based on the  fact that the $\Gamma$-convergence expansion developed in \cite{ADGP1} holds true for several crystalline structures and different types of interactions (see  \cite{DL}), so that the discrete elastic energy induced by the dislocations  can be asymptotically decomposed into the sum of a self energy, concentrated around each dislocation, and of the  renormalised energy, governing the interactions of the limiting singularities.  We start our analysis assuming that the discrete model under consideration exhibits this behavior (see Section~\ref{specific}).
A key argument is an improved lower bound for the discrete energy which accounts for the formation of dipoles which cluster at points that do not appear in the limiting distribution of singularities (see Proposition~\ref{corimp}).
This lower bound guarantees that the crystalline rate dependent dissipation considered in the discrete gradient flow, which in 
general is not continuous with respect to the flat convergence (see Example~\ref{example}), is instead continuous on the proposed discrete dynamics for well prepared initial conditions.
We provide a concise and  almost  self contained presentation, giving references whenever it is  needed, while for further modeling motivations and examples 
we refer the reader to \cite{ADGP2}.

\section{The discrete model}\label{modello}

In this section we introduce the discrete formalism that will be used in the paper (see \cite{AO, ADGP2,ADGP1, DL}).

\subsection{The discrete lattice}
We recall that a {\it Bravais lattice} $\Lambda_B$ in $\R^2$ is a discrete set of points in $\R^2$ generated by two  given linearly independent vectors $v_1,\,v_2$, i.e.,
\begin{equation*}
\Lambda_B:= \left\{ z_1\, v_1+z_2\,v_2, \, z_1, z_2\in \Z \right\}.
\end{equation*}
A {\it complex lattice} $\Lambda_C$ in $\R^2$ is the union of a finite number of translations of a given Bravais lattice $\Lambda_B$, i.e., 
$\Lambda_C$ is of the form
\begin{equation*}
\Lambda_C:= \bigcup_{k=1}^M (\Lambda_B + \tau_k),
\end{equation*}
where $\tau_1, \ldots, \tau_M$ are $M$ given translation vectors in $\R^2$. 
Clearly a Bravais lattice is a particular case of complex lattice (corresponding to $M=1, \, \tau_1 = 0$). In the sequel, we will denote by $\Lambda$ any complex (and in particular Bravais) lattice in $\R^2$. 

\subsection{The reference configuration}\label{boundedlattices}
Let $\Lambda$ be a complex lattice in $\R^2$ and let $\T$ be a periodic triangulation of $\Lambda$, in the sense that if $T\in\T$, then  $v_1+\bT, v_2+\bT\in\T$, where $v_1$ and $v_2$ are the generators of the Bravais lattice associated to $\Lambda$. Such a triangulation always exists. For instance, one can consider the Delaunay tessellation of $\R^2$ associated to $\Lambda$, which is $v_1$ and $v_2$ periodic, and then construct a periodic Delaunay triangulation, starting from a triangulation of a primary domain. We remark that the triangulation $\T$ is not uniquely determined; nevertheless, our analysis will not be affected by the specific choice of $\T$.
From now on, we assume that the triangles $T$ in $\T$ are closed.

Let $\Om\subset \R^2$ be an open bounded set with Lipschitz continuous boundary.
For every $\ep>0$, we denote by $\ep\T$ the family of triangles $\ep T$ with $T\in \T$, we define  the set of the $\ep$-triangular  cells in $\Omega$ as
$$
\Omega^2_{\ep,\Lambda}:=\{T\in\ep\T\,:\,T\subset\bar\Omega\}
$$
and we set $\Omega_{\ep,\Lambda}:=\cup_{T\in\Omega^2_{\ep,\Lambda}}T$.
Moreover, we set $\Omega_{\ep,\Lambda}^0:=\Omega_{\ep,\Lambda}\cap\ep\Lambda$ and 
$$
\Omega_{\ep,\Lambda}^1:=\{(i,j)\in\Omega_{\ep,\Lambda}^0\times\Omega_{\ep,\Lambda}^0\,: i, \, j \in T \text{ for some } T\in \Omega^2_{\ep,\Lambda} \}.
$$ 
Notice that if $\Lambda$ is the equilateral triangular lattice, then $\Omega_{\ep,\Lambda}^1$ is nothing but the class of nearest neighbors in $\Omega^0_{\ep,\Lambda}$.

 In the following we will extend the use of such notations to any given subset $A$ of $\R^2$.

\subsection{Discrete functions and discrete topological singularities}
We denote the class of scalar functions on $\Omega_{\ep,\Lambda}^0$ by
$$
\mathcal{AF}_{\ep,\Lambda}(\Omega):=\left\{u:\Omega_{\ep,\Lambda}^0\to\R\right\}.
$$
In order to introduce the notion of discrete topological singularity, we associate to any bond $(i,j)\in\Omega_{\ep,\Lambda}^1$ an arbitrarily oriented vector $\ell_{i,j}=\ell_{j,i}$ which coincides either with $j-i$ or with $i-j$.

Let $P:\R\to\Z$ be defined as follows
\begin{equation*}
P(t)=\mathrm{argmin}\left\{|t-s|:\,s\in\Z\right\},
\end{equation*}
with the convention that, if the argmin is not unique, then $P(t)$ is the smallest one.
Let $u\in\mathcal{AF}_{\ep,\Lambda}(\Om)$ be fixed. { The discrete plastic strain $\beta_u^{\pl}$ associated to $u$ is defined on the oriented bonds of the triangulation  by  $\beta_{u}^{\pl}(\ell_{m,n})=P(u(n)-u(m))$ if $\ell_{m,n}=n-m$ and $\beta_{u}^{\pl}(\ell_{m,n})=P(u(m)-u(n))$ if $\ell_{m,n}=m-n$.}
Given $T \in \Om_{\ep,\Lambda}^{2}$ and given  a triple $(i,j,k)$ of vertices of $T$ defining a counter-clockwise orientation of $T$, we introduce { the discrete circulation of $u$ around $T$} as  
\begin{equation}\label{alfa}
\alpha_u(T):= \ell_{i,j}\cdot \frac{j-i}{|j-i|^2} \beta^{\pl}_u(\ell_{i,j}) +  \ell_{j,k}\cdot \frac{k-j}{|k-j|^2} \beta^\pl_u(\ell_{j,k})+\ell_{k,i}\cdot \frac{i-k}{|i-k|^2} \beta^\pl_u(\ell_{k,i}).
\end{equation}
{ Notice  that the sign in front of any contribution $\beta_u^{\pl}({\ell_{i,j}})$ depends on the relative orientation of the bond $\ell_{i,j}$ and the counter-clockwise orientation of the triangle $T$.  This ensures that, whenever we sum the  circulation around  two adjacent  triangles, then the contribution on the common bond   cancels. 
}

One can easily check that $\alpha_u$ takes values in the set $\{-1,0,1\}$.  { The values $+1$ and $-1$ for $\alpha_u$  correspond to the presence of a dislocation in the triangle $T$.} Finally, we define the discrete dislocation measure $\mu(u)$ as follows
\begin{equation*}
\mu(u):= \sum_{T \in \Om_{\ep,\Lambda}^2}\alpha_u(T)\delta_{b(T)},
\end{equation*}
where $b(T)$ is the barycenter of the triangle $T$.
By its very definition, for every  subset $A$ of $\Omega$ which is union of $\e$-triangles in $\ep\T$ we have that $\mu(u)(A)$ depends only on the values of $u$ on $\partial A \cap \ep\Lambda$. 

We remark that other variants of the given notion of discrete { circulation} could be adopted. For instance, if $\Lambda$ is a Bravais  lattice, one could define the function $\alpha_u$ on  primitive cells instead of triangles (as done in \cite{ADGP1, DL}), and  the analysis developed in this paper would apply with minor notational changes.

Let $\mathcal{M}(\Omega)$ be the space of Radon measures in $\Omega$. We set 
\begin{equation*}
\begin{split}
& X(\Omega):=\left\{\mu\in\mathcal{M}(\Omega)\,:\ \mu=\sum_{i=1}^{N}d_i\delta_{x_i}\,,\,N\in\N,\,d_{i}=\pm 1, x_i\in\Omega\,, x_i\neq x_j\  \hbox{if } i\neq j\right\}, \\
& X_{\ep,\Lambda}(\Omega):=\left\{\mu\in X(\Omega) : \mu=\sum_{T\in\Omega_{\ep,\Lambda}^2} d{(T)}\,\delta_{b(T)}, \, 
d{(T)} \in \{-1,0,1\}\right\}.
\end{split}
\end{equation*}
We denote by $W^{-1,1}(\Om)$ the dual of $W^{1,\infty}_0(\Om)$, by $\|\cdot\|_{\flt}$ the dual norm in $W^{-1,1}(\Om)$, referred to as {\it flat norm}, and by   $\mu_n \fla \mu$ the flat convergence of $\mu_n$ to $\mu$.


\subsection{The energy functionals}
 Here we introduce a class of energy functionals defined on  $\AF_{\ep,\Lambda}(\Omega)$. 
Let $\{f_{\bj-\bi}\}_{(\bi,\bj)\in (\R^2)^1_{1,\Lambda}}$ be a family of non-negative, continuous, $1$-periodic interaction potentials vanishing on $\Z$, {such that
$$
f_{\bj-\bi}(t)=c_{\bj-\bi}t^2+o(t^2)
$$
for some  constants $c_{\bj-\bi}\ge 0$, for any $(\bi,\bj)\in(\R^2)^1_{1,\Lambda}$.} A prototypical example of these potentials is given by the functions 
 \begin{equation}\label{screw}
 f_{\bj-\bi}(t)=c_{\bj-\bi}\,{\rm dist}^2(t,\Z)\,.
 \end{equation}
 We refer to  \cite{ADGP2}  for a formal derivation of energy potentials of the type \eqref{screw} from  a discrete  anti-plane elasticity model. 

Throughout the paper we will assume the following 
{\it coercivity condition}: for each triangle $T\in\T$ there are at least two distinct bonds for which the interaction potential vanishes only on $\Z$. Precisely, for every $T\in\T$, denoting by $\bi_1$, $\bi_2$, $\bi_3$ its vertices, there exists a permutation $(k_1,k_2,k_3)$ of $(1,2,3)$ such that 
\begin{equation}\label{coerc}
f_{\bi_{k_1}-\bi_{k_2}}(t)>0\ \ \hbox{and}\ \ f_{\bi_{k_1}-\bi_{k_3}}(t)>0\ \quad \forall \ t\not\in \Z.
\end{equation}
 For every $\e>0$, 
the energy functionals 
$F_{\ep,\Lambda}: \mathcal{AF}_{\ep,\Lambda}(\Omega)\to \R$ are defined by

\begin{equation*}
F_{\ep,\Lambda}(u)=\sum_{(i,j)\in \Omega_{\e,\Lambda}^1} f_{\frac {j -i} \ep}(u(j)-u(i))\, .
\end{equation*}

Given $T\in\Omega_{\e,\Lambda}^2$ we  denote by $F_{\ep,\Lambda}(u,T)$  the energy accounting for the interactions between the vertices of the $\ep$-triangle $T$. Precisely, denoting by $i_1$, $i_2$, $i_3$ the  vertices of $T$, 
$$
F_{\ep,\Lambda}(u,T) = \sum_{k,l=1,2,3} f_{\frac {i_l -i_k} \ep}(u(i_l)-u(i_k)).
$$
As a straightforward consequence of the coercivity assumption we get the following lemma.

\begin{lemma}\label{energy-dipole}
There exists a constant $c_0>0$ such that
for any $u\in\AF_{\ep,\Lambda}(\Omega)$  and for any $T\in\Omega_{\ep,\Lambda}^2$
\begin{equation*}
F_{\ep,\Lambda}(u,T)\geq c_0\quad\textrm{whenever } \alpha_u(T)\neq 0.
\end{equation*}
\end{lemma}

\begin{proof}
Let $(i_1,i_2,i_3)$ denote a counter-clockwise  oriented  triple of  vertices of $T$.
By the very definition  \eqref{alfa} of $\alpha_u$ and  by the triangular inequality, we have that if $ \alpha_u(T)\neq 0$, then
\begin{equation}\label{coercivo}
\sum_{k,l=1,2,3}\di(u(i_l)-u(i_k),\Z)\ge |\alpha_u(T)|= 1.
\end{equation}
Since
$\di(\cdot,\Z)\le\frac 1 2$, it follows that at least two addenda of the sum in \eqref{coercivo} are not smaller than $\textstyle \frac 1 4$ and hence the claim follows by the coercivity assumption \eqref{coerc}.
\end{proof}

It is convenient to express the energy in terms of the dislocation measure. More precisely, given $\mu \in X_{\e,\Lambda}(\Om)$, we set 
\begin{equation*}
\F_{\ep,\Lambda}(\mu):= \inf \{F_{\e,\Lambda}(u)\,: \, u\in \mathcal{AF}_{\ep,\Lambda}(\Omega),\, \mu(u) = \mu \}.
\end{equation*}

In what follows, we also need  a localised version of the energy functionals above: for any  subset $A$ of $\Omega$, we set
\begin{eqnarray*}
&&F_{\ep,\Lambda}(u,A):=\sum_{(i,j)\in A_{\e,\Lambda}^1} f_{\frac {j -i} \ep}(u(j)-u(i)),\\
&&\F_{\ep,\Lambda}(\mu,A):= \inf \{F_{\e,\Lambda}(u,A)\,: \, u\in \mathcal{AF}_{\ep,\Lambda}(A),\, \mu(u) = \mu \}.
\end{eqnarray*}
\section{Renormalised energy and $\Gamma$-convergence assumption} \label{specific}

This section is devoted to the statement of the $\Gamma$-convergence assumption for the functionals $\F_{\ep,\Lambda}$. To this purpose, we first define the renormalised energy in our setting.

The  (isotropic) renormalised energy (see formula \eqref{WI} below) has been introduced in \cite{BBH} (see also \cite{SS2}) as the interaction energy between vortices in the Ginzburg-Landau framework. In \cite{ADGP1} it has been shown  that such a renormalised energy also  governs  the interactions between screw dislocations in isotropic anti-plane elasticity.
{ However, the continuous counterpart of the discrete energy $\F_{\ep,\Lambda}$  is in general anisotropic as well as the corresponding renormalised energy.
In fact, in absence of defects (i.e., for $\mu(u)=0$) and under suitable assumptions on the potentials $f_{\bj-\bi}$, the discrete elastic energy 
$\Gamma$-converges 
(see \cite{AC})    to a continuous energy of the form
\begin{equation*}
\int_\Omega Q\nabla u\nabla u\ud x\qquad \forall u\in H^1(\Omega),
\end{equation*}
being $Q$ a symmetric positive definite matrix determined by $f_{\bj-\bi}$. As a consequence the corresponding renormalised energy can be formally computed by means of a change of variable.}

\subsection{Renormalised energy.} Let $U$ be an open bounded subset of $\R^2$ with Lipschitz continuous boundary and let $\nu:= \sum_{i=1}^M d_i \delta_{y_i}$, with $M\in\N$, $d_i\in\{-1,+1\}$, $y_i\in U$ and $y_i\neq y_j$ for $i\neq j$. Let $R_{\nu,U}:U\to\R$ be the solution of
\begin{equation}\label{Rf}
\left\{\begin{array}{ll}
\Delta R_{\nu,U}(y)=0&\textrm{ in } U,\\
R_{\nu,U}(y)=-\sum_{i=1}^M d_i\log|y-y_i|&\textrm{ on }\partial U.
\end{array}\right.
\end{equation} 
The (isotropic) renormalised energy  (see \cite{BBH}) is defined by
\begin{equation}\label{WI}
\mathscr{W}_{U}(\nu):=-{\pi}\sum_{i\neq j}d_i\,d_j\log|y_i-y_j|-{\pi}\sum_{i=1}^M d_i R_{\nu,U}(y_i).
\end{equation}
We set $\W_{I,U}(\nu):=\frac{1}{2\pi^2}\mathscr{W}_U(\nu)$ and 
for any open subset $\Om'\subseteq\Om$ with Lipschitz continuous boundary and for any $\mu=\sum_{i=1}^Md_i\delta_{x_i}\in X(\Om)$ with $\supp\mu\subset\Omega'$ we set
\begin{equation}\label{WQ}
\W_{Q,\Omega'}(\mu):=\lambda_{Q}\W_{I,\tilde Q^{-\frac 1 2}(\Omega')}(\tilde Q^{-\frac 1 2}_\sharp \mu),
\end{equation}
where 
$$
\lambda_{Q}:=\sqrt{\det Q}, \quad \quad \tilde Q:=\frac{Q}{\sqrt{\det Q}},
$$
and
$\tilde Q^{-\frac 1 2}_\sharp(\mu)$  is the push-forward of $\mu$ through $\tilde Q^{-\frac 1 2}$ defined by 
$$
\tilde Q^{-\frac 1 2}_\sharp\mu:=\sum_{i=1}^Md_i\delta_{\tilde Q^{-\frac 1 2} x_i}.
$$
To ease the notations, we set $\W_{Q}(\mu):=\W_{Q,\Omega}(\mu)$.

Throughout the paper we will assume that the  following $\Gamma$-convergence result holds true.

{\bf $\Gamma$-convergence assumption:} There exist a symmetric positive matrix $Q$ and a constant $\gamma >0$ such that the following holds.

\begin{itemize}
\item[(i)] (Compactness) Let $M\in\N$ and let $\{\mu_\ep\}\subset X_{\ep,\Lambda}(\Omega)$ be a sequence satisfying $\F_{\ep,\Lambda}(\mu_\ep)-M\frac{\lambda_{Q}}{2\pi}|\log\ep|\le C$. 
Then, up to a subsequence, $\mu_\ep\fla\mu$ for some $\mu=\sum_{i=1}^N d_i\delta_{x_i}$ with $d_i\in\Z\setminus\left\{0\right\}$, $x_i\in\Omega$, $x_i\neq x_j$ for $i\neq j$  and $\sum_{i}|d_i|\le M$. Moreover, if $\sum_{i}|d_i|=M$, then $N=M$, i.e., $|d_i|=1$ for any $i$. 
\item[(ii)] ($\Gamma$-$\liminf$ inequality)  Let $\{\mu_\ep\}\subset X_{\ep,\Lambda}(\Omega)$ be  such that $\mu_\ep\fla\mu\in X(\Omega)$. Then,
\begin{equation}\label{liminfgamma1T}
\liminf_{\ep\to 0}\F_{\ep,\Lambda}(\mu_\ep)-|\mu|(\Omega) \frac{\lambda_{Q}}{2\pi} |\log\ep|\ge \W_{Q}(\mu)+|\mu|(\Omega)\,\gamma.
\end{equation}\item[(iii)] ($\Gamma$-$\limsup$ inequality) Given $\mu\in X(\Omega)$, there exists $\{\mu_\ep\}\subset X_{\ep,\Lambda}(\Omega)$ with $\mu_\ep\fla\mu$ such that
\begin{equation}\label{gammalimsup}
\F_{\ep,\Lambda}(\mu_\ep)-|\mu|(\Omega) \frac{\lambda_{Q}}{2\pi} |\log\ep|\to\W_{Q}(\mu)+|\mu|(\Omega)\gamma.
\end{equation}
\end{itemize}

In \cite{ADGP1, DL} it has been shown that,  under suitable conditions on the potentials $f_{\bj-\bi}$, the $\Gamma$-convergence assumption above  holds true, and $Q$ and $\gamma$ are explicitly determined by $f_{\bi-\bj}$ and $\Lambda$. In particular, if $\Lambda$ is the square lattice and $f_{\bi-\bj}(t)=\di^2(t,\Z)$, then $Q=I$ and the corresponding renormalised energy is  given by $\mathbb{W}_{I,\Omega}(\mu)$.

\begin{remark}
\rm{ Note that if $\{\mu_\ep\}\subset X_{\ep,\Lambda}(\Omega)$ is  such that $\mu_\ep\fla\mu\in X(\Omega)$, then,  for every open subset $\Omega'\subset\Omega$ with Lipschitz continuous boundary and  with $\supp\mu\subset\Omega'$, there holds $\mu_\ep\fla\mu\in X(\Omega')$. Therefore, by \eqref{liminfgamma1T}, it immediately follows that
\begin{equation}\label{liminfgamma1Tloc}
\liminf_{\ep\to 0}\F_{\ep,\Lambda}(\mu_\ep,\Omega')-|\mu|(\Omega) \frac{\lambda_{Q}}{2\pi} |\log\ep|\ge \W_{Q,\Omega'}(\mu)+|\mu|(\Omega)\gamma.
\end{equation}
}
\end{remark}

\bigskip

The rest of this section is devoted to the proof of some properties which are consequences of the $\Gamma$-convergence assumption and that will be useful in the sequel.

\begin{lemma}\label{ContRen}
Let $U$ be an open bounded subset of $\R^2$ with Lipschitz continuous boundary,  $g\in W^{1,\infty}(U)$, and $\xi_1,\ldots,\xi_N\in \overline{U}$, with $N\in \N$. For any $\rho\geq 0$,  set $U_\rho:= U\setminus \cup_{i=1}^N B_\rho(\xi_i)$ (in particular, $U_0=U$);  let $R_\rho$ be the harmonic function in $U_\rho$ satisfying $R_\rho-g\in H^1_0(U_\rho)$.
Then, $R_\rho$ converges to $R_0$ as $\rho\to 0$ locally uniformly in $U\setminus\cup_{i=1}^N\{\xi_i\}$.
\end{lemma}

\begin{proof}
We first notice that $R_0\in C^\infty(U)\cap C(\overline U)$, since  $R_0$ is harmonic in $U$  and $\partial U$ is Lipschitz continuous.

Set $\psi_\rho:=R_\rho-R_0$. Trivially, $\psi_\rho$ solves the following minimisation problem
\begin{equation}\label{psirho}
\min\left\{ \int_{U_\rho}|\nabla \psi|^2 \ud x\,:\ \psi-(g-R_0)\in H^1_0(U_\rho) \right\}\,.
\end{equation}
By  linearity it is enough to consider the case $N=1$, i.e., with a single point $\xi_1=\xi$.

Set $D:={\rm diam} (U)$ and let $\Phi_\rho$ be the capacitary potential of $B_\rho(\xi)$ in $B_D(\xi)$, namely, $\Phi_\rho(x):=\frac{\log |x-\xi|-\log D}{\log \rho-\log D}$. As $\rho\to 0$, we have that $\Phi_\rho\to 0$ in $H^1(B_D(\xi))$ and also pointwise in $B_D(\xi)\setminus \{\xi\}$. Set $\hat\psi_\rho:=(g-R_0)\Phi_\rho$.
 Since $0\leq \Phi_\rho(x)\leq 1$, with $\Phi_\rho=1$ on $\partial B_\rho(\xi)$, and $g-R_0\in H^1_0(U)\cap L^\infty(U)$, it is immediate to check that $\hat\psi_\rho\to 0$ in $H^1(U)$, and $\hat\psi_\rho-(g-R_0)\in H^1_0(U_\rho)$. It follows that $\hat\psi_\rho$ is a competitor for the problem \eqref{psirho} and hence  
\begin{equation}\label{convgrad}
\int_{U_\rho}|\nabla\psi_\rho|^2\ud x\to 0.
\end{equation}
Moreover, since $\psi_\rho =0$ on $\partial U\setminus B_\rho(\xi)$,  \eqref{convgrad} combined with Poincar\'e inequality implies that $\psi_\rho\to 0$ in $H^1_{\loc}(U\setminus\{\xi\})$ as $\rho\to 0$.  Since the functions $\psi_\rho$ are harmonic, they also converge locally uniformly to zero on $U\setminus\{\xi\}$,  and hence the claim follows by the very definition of $\psi_\rho$.
\end{proof}

For any $\nu=\sum_{i=1}^K z_i \delta_{\xi_i}$ with $K\in\N$, $z_i\in \Z\setminus \{0\}$, $\xi_i\in \overline\Om$ and for any $r>0$, we set 
\begin{equation*} 
\Om_r(\nu):= \Om \setminus \cup_{i=1}^K B_{r}(\xi_i). 
\end{equation*}
Notice that Lemma \ref{ContRen} holds true also if we replace the set $U_\rho$ with $\tilde U_\rho:= U\setminus \cup_{i=1}^N \{\xi_i + \rho E\}$, where E is a fixed set with Lipschitz continuous boundary. 
Therefore, recalling \eqref{Rf}, \eqref{WI} and  \eqref{WQ} by means of a change of variable we deduce the following Corollary.

\begin{corollary}\label{cor}
Let  $\nu=\sum_{i=1}^N z_i \delta_{\xi_i}$ with $N\in\N$, $z_i\in \Z\setminus \{0\}$, $\xi_i\in \overline\Om$ and let $\mu\in X(\Om)$ be such that $\supp \mu \cap\supp \nu=\emptyset$.
Then,  $\W_{Q}(\mu, \Om_{\rho}(\nu)) \to  \W_{Q}(\mu,\Om)$ as $\rho\to 0$.
\end{corollary}

The following proposition provides an improved lower bound for the energy of a sequence $\mu_\ep$ converging flat to $\mu$. This lower bound accounts for the presence of dipoles of $\mu_\ep$ which do not cluster in the support of $\mu$.

\begin{proposition}\label{corimp}
Let $\{\mu_\ep\}\subset X_{\ep,\Lambda}(\Omega)$ be  such that $\mu_\ep\fla\mu\in X(\Omega)$. Assume that
\begin{eqnarray}\label{assu}
\F_{\ep,\Lambda}(\mu_\ep)-M\frac{\lambda_{Q}}{2\pi}|\log\ep|\leq C\,,
\end{eqnarray}
where $M= |\mu|(\Om)$. Then there exist $N,L\in\N$ and a measure $\nu=\sum_{j=1}^N 2\delta_{y_j}+\sum_{k=1}^L\delta_{z_k}$, with $y_j\in\Om\setminus\supp\mu$ and $z_k\in\partial\Omega$, not necessarily distinct, such that, up to a subsequence
\begin{equation}
|\mu_\ep|\res{A}\weakstar \nu\quad\mbox{in } A,
\end{equation}
where $A:=\R^2\setminus\supp\mu$. 
Moreover,
\begin{equation}\label{liminfgamma2T}
\liminf_{\ep\to 0}\F_{\ep,\Lambda}(\mu_\ep)-M\frac{\lambda_{Q}}{2\pi}|\log\ep|\ge \W_{Q}(\mu)+M\gamma +c_0\nu(A),
\end{equation}
where $c_0$ is given by Lemma~\ref{energy-dipole}.
In particular, if $\{\mu_\ep\}$ satisfies \eqref{gammalimsup}, then $\nu= 0$.
\end{proposition}

\begin{proof}
By \eqref{liminfgamma1T} (see also \eqref{liminfgamma1Tloc}) and \eqref{assu} we have that, for any $\sigma>0$, 
\begin{equation}\label{limro}
\F_{\ep,\Lambda}(\mu_\ep, \Omega_\sigma(\mu))\leq C_\sigma \quad \text{ for some } C_\sigma>0,
\end{equation} 
whence, by Lemma \ref{energy-dipole}, we get $|\mu_\ep|(\Omega_\sigma(\mu))\leq \frac{C_\sigma}{c_0}$. Moreover, since  $\mu_\ep\fla\mu\in X(\Omega)$, we have that  $\mu_\ep\res{\Omega_\sigma(\mu)}\fla 0$  in $\Omega_\sigma(\mu)$. It follows that, up to a subsequence, $|\mu_\ep|\res{A}\weakstar \nu$ in $ A$ for some $\nu=\sum_{j\in J}2\delta_{y_j}+\sum_{k=1}^L\delta_{z_k}$, with $J\subseteq \N$, $L\in\N$, $y_j\in\Om\setminus\supp\mu$ and $z_k\in\partial\Omega$. By \eqref{limro}, any cluster point of $\{y_j\}$ 
belongs to $\supp \mu$. 

Let $\sigma >0$ be fixed and let $\nu_\sigma:= \nu\res(\Om_\sigma(\mu)\cup\partial \Om)$. Now, fix $\rho>0$. Since $\mu_\ep \res \Om_{\rho}(\nu_\sigma) \fla \mu$ in  $\Om_{\rho}(\nu_\sigma)$, by \eqref{assu}, \eqref{liminfgamma1Tloc} and Lemma \ref{energy-dipole}, we have
\begin{multline*}
C\ge \liminf_{\e\to 0} \F_{\ep,\Lambda}(\mu_\ep) - M \frac {\lambda_{Q}}{2\pi}|\log\ep|  
\\
\ge \liminf_{\e\to 0} \F_{\ep,\Lambda}(\mu_\ep, \Om_{\rho}(\nu_\sigma)) - M \frac {\lambda_{Q}}{2\pi}|\log\ep| + \liminf_{\e\to 0} \F_{\ep,\Lambda}(\mu_\ep, \Om \setminus \Om_{\rho}(\nu_\sigma))
\\
\ge
\W_{Q}(\mu, \Om_{\rho}(\nu_\sigma)) + M\gamma + c_0 \nu_\sigma(A).
\end{multline*}
By applying Corollary \ref{cor} with $\nu = \nu_\sigma$ 
we have that $\W_{Q}(\mu, \Om_{\rho}(\nu_\sigma)) \to  \W_{Q}(\mu,\Om)$ as $\rho\to 0$ and
we obtain \eqref{liminfgamma2T} by first letting $\rho\to 0$, and then $\sigma\to 0$.
Finally, by \eqref{liminfgamma2T}, we have immediately that $\nu(A)$ is finite.
\end{proof}

\section{Discrete gradient flow of  $\mathcal F_{\ep,\Lambda}$}\label{secMMW1}
In this section we  introduce  the minimising movement scheme,  referred to as {\it discrete gradient flow}, for the energy  $\mathcal F_{\ep,\Lambda}$.
As mentioned in the introduction,  such a scheme is governed by a dissipation  
that  accounts  for the specific glide directions of the  crystal.
 \subsection{Dissipations}
Here we introduce the class of rate dependent dissipation functionals, that will measure the energy spent to move a configuration of dislocations during the discrete gradient flow. 
If the dislocation configuration at two different time steps $t_1$, $t_2$ is given by a single Dirac mass $\delta_{{x(t)}}$ centered at $x(t)$, then we assume that the energy spent to move the dislocation from $x(t_1)$ to $x(t_2)$ can be expressed as $\phi^2(x(t_1) - x(t_2))$, for a suitable norm $\phi$. In order to account for the glide directions of the dislocations $\phi$ must be chosen to be minimal on a finite set of directions, i.e., it is a crystalline norm.

We define the dissipation for general configurations of Dirac masses  in two steps. First assume that all the dislocations have the same sign. More precisely, let $\nu_1=\sum_{i=1}^{N_1}d_i^1\delta_{x_i^1}$ and $\nu_2=\sum_{j=1}^{N_2}d_j^2\delta_{x_j^2}$ with $d_i^1,d_j^2\in\N$ for every $i=1,\ldots,N_1$ and $j=1,\ldots,N_2$  and set
\begin{multline*} 
\disstilde(\nu_1,\nu_2):=\min\left\{\sum_{l=1}^L \phi^2( q_l-p_l)\,:\,L\in\N,\,q_l\in\supp \nu_1\cup\partial\Omega,\right.\\
\left.\quad \,p_l\in\supp \nu_2\cup\partial\Omega,\,\sum_{l=1}^{L}\delta_{q_l}\res\Omega=\nu_1,\sum_{l=1}^L\delta_{p_l}\res\Omega=\nu_2\right\}.
\end{multline*}
Note that here we optimise among all possible connections between points in the support of $\nu_1$ and points in the support of $\nu_2$ (all counted with their multiplicity), possibly including connections with points at the boundary of $\Omega$. From the very definition of $\disstilde$ one can easily check that
\begin{equation}\label{triangolare}
\disstilde(\nu_1+\rho_1,\nu_2+\rho_2)\leq \disstilde(\nu_1,\nu_2)+ \disstilde(\rho_1,\rho_2) ,
\end{equation}
for any measures $\rho_1$ and $\rho_2$ which are sums of positive Dirac masses.
    
For the general case of $\mu_1=\sum_{i=1}^{N_1}d_i^1\delta_{x_i^1}$ and $\mu_2=\sum_{i=1}^{N_2}d_i^2\delta_{x_i^2}$
with $d_i^1,d_i^2\in\Z$ we set
\begin{equation*} 
\diss(\mu_1,\mu_2):=\disstilde(\mu_1^++\mu_2^-,\mu_2^++\mu_1^-),
\end{equation*}
where $\mu_j^+$ and $\mu_j^-$ are the positive and the negative part of  $\mu_j$.
 
By standard arguments in optimal transport theory \cite{V} (see for instance \cite[formula (6.3)]{ADGP1}), one can easily prove that 
there exists a positive constant $C_{\phi,\Omega}$ such that 
\begin{equation}\label{dissflat}
\diss(\mu_1,\mu_2)\le C_{\phi,\Omega}\|\mu_2-\mu_1\|_{\flt}\quad\textrm{for any } \mu_1,\mu_2\in X(\Omega).
\end{equation}
 \begin{lemma}
   Let $\mu_\ep,\nu_\ep\in X(\Omega)$ be such that $\mu_\ep\fla\mu$ and $\nu_\ep\fla\nu$, for some $\mu,\nu\in X(\Omega)$. Then
\begin{equation}\label{contprima}
\limsup_{\ep\to 0} \diss(\mu_\ep,\nu_\ep)\le\diss(\mu,\nu)\,.
\end{equation}
\end{lemma}
\begin{proof}
Set $\bar\mu_\ep:=\mu_\ep-\mu$ and $\bar\nu_\ep:=\nu_\ep-\nu$; then,  by the very definition of $\diss$ (and of $\disstilde$), using \eqref{triangolare} and \eqref{dissflat}, we have
\begin{multline*}
\diss(\mu_\ep,\nu_\ep)=\diss(\mu+\bar\mu_\ep,\nu+\bar\nu_\ep)
=\disstilde(\mu^+ +\bar\mu_\ep^+ +\nu^- + \bar\nu_{\ep}^-, \mu^- +\bar\mu_\ep^- +\nu^++ \bar\nu_{\ep}^+)\\
\le \disstilde(\mu^+ +\nu^-,\mu^-+\nu^+)+\disstilde(\bar\mu_\ep^+ +\bar\nu^-_\ep,\bar\mu_\ep^-+\bar\nu_\ep^+)
=\diss(\mu,\nu)+\diss(\bar\mu_\ep,\bar\nu_\ep)\\
\le\diss(\mu,\nu)+C_{\phi,\Omega}\|\bar\mu_\ep-\bar\nu_\ep\|_{\flt}\le \diss(\mu,\nu)+C_{\phi,\Omega}(\|\mu-\mu_\ep\|_\flt+\|\nu-\nu_\ep\|_{\flt});
\end{multline*}
since $\mu_\ep\fla\mu$ and $\nu_\ep\fla\nu$ as $\ep\to 0$, we get \eqref{contprima}.
 \end{proof}

In the following, given $\mu_\e, \, \mu\in X(\Om)$, the notation ${\supp} \mu_\ep\buildrel{{\mathcal H}}\over{\longrightarrow}\supp\mu$ means that the support of $\mu_\e$ converges, in the Hausdorff sense, to the support of $\mu$.  A key point in our analysis is the following continuity property of the dissipation.
\begin{lemma}\label{lemmaimp}
Let $M\in\N$, $d_1,\ldots,d_M\in\{-1,+1\}$ and let $\{x_1,\ldots,x_M\},\,\{y_1,\ldots,y_M\}\subset\Omega$ be such that for every $i=1, \ldots , M$
\begin{equation}\label{distass}
|x_i-y_i|<\min_{i\neq j}\{|x_i-x_j|,\, |y_i-y_j|,\,|x_i-y_j|,\,\di(x_i,\partial\Omega),\,\di(y_i,\partial\Omega)\}.
\end{equation}
Let $\mu:=\sum_{i=1}^M d_i\delta_{x_i}$ and $\nu:=\sum_{i=1}^Md_i\delta_{y_i}$ and let 
 $\mu_\ep,\nu_\ep\in X_{\ep,\Lambda}(\Omega)$ be such that $\mu_\ep\fla\mu$ and $\nu_\ep\fla\nu$. Moreover assume that 
\begin{equation}\label{supporto}
{\supp} \mu_\ep\buildrel{{\mathcal H}}\over{\longrightarrow}\supp\mu\qquad {\supp} \nu_\ep\buildrel{{\mathcal H}}\over{\longrightarrow}\supp\nu\,. 
\end{equation}
Then
\begin{equation*} 
\lim_{\ep\to 0} \diss(\mu_\ep,\nu_\ep)=D_{\phi}(\mu,\nu)=\sum_{i=1}^M\phi^2(x_i-y_i)\,.
\end{equation*}
\end{lemma} 
\begin{proof}
By \eqref{supporto} we have that, for any $\rho>0$ and for $\ep$ small enough (depending on $\rho$),
${\supp}\mu_\ep\subset\cup_{i=1}^M B_{\rho}(x_i)$ and ${\supp}\nu_\ep\subset\cup_{i=1}^M B_{\rho}(y_i)$. Moreover, since $\mu_\ep\fla\mu$ and $\nu_\ep\fla\nu$, it is easy to prove that, for all $\rho$ small enough,  $\mu_\ep(B_\rho(x_i))=\nu_\ep(B_\rho(y_i))=d_i$ for any $i=1,\ldots,M$ and for $\ep$ small enough (depending on $\rho$). 

Therefore, any connection between $\mu_{\ep}$ and $\nu_{\ep}$ contains at least $M$ segments joining some point in $B_\rho(x_i)$ (for any $i=1,\ldots,M$) with some other point which either lies on $\partial \Om$ or belongs to some $B_\rho(y_j)$ for some $j\in1,\ldots,M$. In view of \eqref{distass}, we have
$$
\liminf_{\e\to 0} \diss(\mu_\ep,\nu_\ep)\geq \sum_{i=1}^M\phi^2(x_i-y_i) - C\rho = D_{\phi}(\mu,\nu) - C\rho\,,
$$ 
for some constant $C>0$ independent of $\rho$. The conclusion follows by the arbitrariness of $\rho$, together with  the upper bound \eqref{contprima}.
\end{proof}

 \begin{example}\label{example}{\rm
 It is easy to see that if the assumptions on $\mu$ and $\nu$, and $\mu_\ep$ and $\nu_\ep$, are not satisfied in general the dissipation $\diss$ is not continuous with respect to the flat convergence.
 Indeed, { assuming by simplicity $\Omega=\R^2$}, if (\ref{supporto}) is not satisfied it is enough to consider the case $\mu=\delta_{(0,0)}$ and $\nu=\delta_{(1,0)}$, which can be approximate in the flat norm by two sequences $\mu_n=\mu$ and $\nu_n=\delta_{(1,0)}+\delta_{(\frac{1}{2}-\frac{1}{n},0)}-\delta_{(\frac{1}{2}+\frac{1}{n},0)}$. It is immediate to check that $\diss(\mu,\nu)=1$ while $\lim_{n\to\infty}\diss(\mu_n,\nu_n)=\frac{1}{2}$.
 
 A counterexample to the continuity if \eqref{distass} is not satisfied is given by $\mu=\delta_{(1,0)}+\delta_{(1,1)}$ and $\nu=\delta_{(0,0)}+\delta_{(-1,0)}$. This pair of measures can be approximated in the flat norm, satisfying \eqref{supporto}, by $\mu_n=\mu$ and $\nu_n=\nu+\delta_{(0,-\frac{1}{n})}-\delta_{(0, \frac{1}{n})}
$. In this case $\diss(\mu,\nu)=6$ while  $\lim_{n\to\infty}\diss(\mu_n,\nu_n)=4$.}
 \end{example}
\begin{corollary}\label{contrecov}
Let $\mu,\nu\in X(\Omega)$ be as in Lemma~\ref{lemmaimp} and let $\{\mu_\ep\},\{\nu_\ep\}\subset X_{\ep,\Lambda}(\Omega)$ be such that $\mu_\ep\fla\mu$, $\nu_\ep\fla\nu$. If $\{\mu_\ep\}$ and $\{\nu_\ep\}$ satisfy \eqref{gammalimsup}, then
\begin{equation}
\lim_{\ep\to 0}\diss(\mu_\ep,\nu_\ep)=\diss(\mu,\nu).
\end{equation}
\end{corollary}

\begin{proof}
By assumption, in view of Proposition \ref{corimp}, we have that $\{\mu_\ep\}$ and $\{\nu_\ep\}$ satisfy \eqref{supporto}. Then the conclusion follows from Lemma \ref{lemmaimp}. 
\end{proof}
\subsection{The minimising movement scheme}
Here we  introduce the discrete gradient flow of $\F_{\ep,\Lambda}$ with respect to the dissipation $\diss$. To this purpose, first notice that, by the $\Gamma$-convergence assumption \eqref{liminfgamma1T}, a given dipole $\delta_x-\delta_y$ induces an elastic energy which blows up as $|\log\e|$ as $\e\to 0$. 
It is  clear that, for $\e$ small enough, it is always convenient to annihilate (in a single time step) such a pair of dislocations, paying a finite dissipated energy (independent of $\e$) while gaining an amount of elastic energy of order $|\log\e|$. It is therefore convenient to look at local minimisers instead of global ones during the step by step minimising movements. To this purpose, we introduce a length scale  $\delta$, and we look for $\delta$-close minimisers. While it is essential to fix such a length scale, it turns out that its specific choice does not affect at all the dynamics: dislocations move with finite velocity, and hence at each time step they make a ``jump'' of the same order of the time step $\tau$, which for small $\tau$ is smaller than any fixed $\delta$. 
\begin{definition}\label{MMepsDef}
Fix $\delta>0$ and let $\ep,\tau>0$. Given $\mu_{\ep,0} \in X_{\ep,\Lambda}(\Om)$,   we say that $\{\mu_{\ep,k}^\tau\}$, with $k\in\N\cup \{0\}$, is a solution of the discrete gradient flow of $\F_{\ep,\Lambda}$ from $\mu_{\ep, 0}$ if $\mu_{\ep,0}^\tau=\mu_{\ep,0}$, and for any $k\in\N$, $\mu_{\ep,k}^\tau$  satisfies
\begin{equation}\label{MMeps}
\begin{split}
&\mu_{\ep,k}^{\tau}\in\mathrm{argmin}\left \{\F_{\ep,\Lambda}(\mu) + \frac{\diss(\mu,\mu^{\tau}_{\ep,k-1})}{2\tau} : \mu\in X_{\ep,\Lambda}(\Om),\right.\\
&\left.\qquad \phantom{\mu_{\ep,k}^{\tau}\in\mathrm{argmin}\,\F_{\ep,\Lambda}(\mu) + \frac{\diss(\mu,\mu^{\tau}_{\ep,k-1})}{2\tau}}\,\|\mu -\mu_{\ep,k-1}^\tau\|_{\flt}\le\delta\right\}.
\end{split}
\end{equation}
\end{definition} 
Notice that the existence of a minimiser is obvious, since $\mu$ lies in $X_{\ep,\Lambda}(\Om)$ which is a finite set. 
\section{Asymptotic dynamics}
In this section we analyse the limit as  $\ep,\tau\to0$ of the discrete gradient flow introduced in Definition \ref{MMepsDef}. 

\begin{definition}\label{wpid}
{ Given $\mu_0\in X(\Omega)$, we say that  $\{\mu_{\ep,0}\}\subset X_{\ep,\Lambda}(\Om)$ are \emph{well prepared initial conditions with respect to $\mu_0$} if  
it holds}
\begin{equation*} 
\mu_{\ep,0} \fla \mu_0, \qquad  \lim_{\ep \to 0} \F_{\ep,\Lambda} (\mu_{\ep,0})-M\,\frac{\lambda_{Q}}{2\pi}\,|\log \ep| = \W_{Q}(\mu_0)+M\,\gamma.
\end{equation*}
\end{definition}
Given $\mu_0:= \sum_{i=1}^{M}d_{i,0}\delta_{x_{i,0}}\in X(\Omega)$ we set ${d_0}:=(d_{1,0},\ldots,d_{M,0})\in\{-1,+1\}^M$, $x_0 := (x_{1,0},\ldots,x_{M,0})\in\Omega^M$, and 
\begin{equation}\label{r_0}
r_0:=\min\{\di_{i\ne j}(x_{i,0},x_{j,0}),\di(x_{i,0},\partial\Omega)\}.
\end{equation}
For any  $x=(x_1,\ldots,x_M)\in\Omega^M$, we define $W_{\Lambda,{d_0}}(x):=\W_{Q}(\sum_{i=1}^Md_{i,0}\delta_{x_i})$.
In this section we will prove the following result.

\begin{theorem}\label{prova}
{ Let $\mu_0\in X(\Om)$. For any $0<r<r_0$, there exists $\delta_r>0$ such that for any  $0<\delta<\delta_r$  the following holds true:
Let $\{\mu_{\ep,0}\}$ be well prepared initial conditions with respect to $\mu_0$, let $\tau>0$ and let  $\{\mu_{\ep,k}^\tau\}$ be a solution of the discrete gradient flow of $\mathcal{F}_{\ep,\Lambda}$ from $\mu_{\ep,0}$. Then we have:}

\begin{itemize}
\item[(i)] \emph{(Limit as $\ep\to 0$)} For any $k\in\N$ there exists $\mu_k^\tau\in X(\Omega)$ with $|\mu_k^\tau|(\Om)\le M$, such that, up to subsequences, $\mu_{\ep,k}^\tau\fla\mu_k^\tau$ as $\ep\to 0$. 
Moreover, there exists $k_r^\tau\in\N$ with
\begin{equation}\label{tempofinale}
T_r:=\liminf_{\tau\to 0}k_r^\tau \tau>0,
\end{equation}
such that for any $k=1,\ldots,k_r^\tau$
$$
\mu_k^\tau=\sum_{i=1}^Md_{i,0}\delta_{x_{i,k}^\tau},
$$ 
for some $x^\tau_k:=(x_{1,k}^\tau,\ldots,x_{M,k}^\tau)\in\Omega^M$
with
$$
\min\{\di_{i\neq j}(x^\tau_{i,k},x^\tau_{j,k}),\di(x^\tau_{i,k},\partial\Omega)\}\ge r.
$$

\item[(ii)] \emph{(Limit as $\tau\to 0$)}  Up to a subsequence,  the piecewise affine interpolation in time $x^\tau(t)$ of $x^\tau_{k}$  converges, as $\tau\to 0$, locally uniformly in $[0, T_{r})$ to a  Lipschitz continuous map  $x:[0, T_{r})\to\Omega^M$ which solves the differential inclusion
\begin{equation} \label{diffinc_prova}
\begin{cases}
\dot{x}_i(t)\!\in \!\partial^{-}(\frac{\phi^2}{2})^*\left(-\nabla_{x_i} \!W_{\Lambda,{d_0}}(x(t))\right)\,\text{for }i=1,\ldots,M,\,\textrm{for a.e. } t\in[0, T_{r})\\ 
x(0)=x_0.
\end{cases}
\end{equation}
\end{itemize}
\end{theorem}

In \eqref{diffinc_prova}, $\partial^-(\frac{\phi^2}{2})^*$ denotes the subdifferential of the polar function $(\frac{\phi^2}{2})^*$ defined by
\begin{equation*} 
\left(\frac{\phi^2}{2}\right)^*(\xi):= \max_{\eta\in\R^2} <\xi,\eta> - \frac{\phi^2}{2}(\eta) = 2(\phi^2)^*(\xi).
\end{equation*} 
To ease the notations, we set
\begin{equation}\label{Phi}
\Psi(\cdot):=\left(\frac{\phi^2}{2}\right)^{*}(\cdot).
\end{equation}

The proof of Theorem \ref{prova} will be a consequence of Theorem \ref{MM1} and Theorem \ref{thmtautozero}, where the thresholds $\delta_r$, $k_r^\tau$ and $T_r$ are explicitly defined. In the former we prove that a solution of the discrete gradient flow of the energy $\F_{\ep,\Lambda}$ converges as $\ep\to 0$ to a solution of the {\it discrete gradient flow of  the renormalised energy $W_{\Lambda,{d_0}}$}, which is a generalised implicit Euler scheme for $W_{\Lambda,{d_0}}$ (see Definition \ref{defeuimp}); in the latter we show that such a solution tends as $\tau\to 0$ to a solution of  problem \eqref{diffinc_prova}. 

Actually,  it turns out that the final time step $k_r^\tau$ in Theorem \ref{prova} depends on the measures $\mu_k^\tau$.
As a consequence, although  in the statement of Theorem \ref{prova}, we first send $\ep\to 0$ and afterwards $\tau\to 0$,  it is convenient first to analyse the discrete in time gradient flow of the renormalised energy and its limit as $\tau\to 0$, and, then, the limit as $\ep\to 0$ of the discrete gradient flow of $\F_{\ep,\Lambda}$. 
\subsection{The minimising movement scheme for the renormalised energy}
\vskip1mm\qquad

\noindent Fix  initial conditions  ${d_0}:=(d_{1,0},\ldots,d_{M,0})\in\{-1,+1\}^M$, with
$M\in\N$, and
$x_0 = (x_{1,0},\ldots,x_{M,0})\in\Omega^M
$, with $x_{i,0}\neq x_{j,0}$ for $i\neq j$,
and let $r_0$ be as in \eqref{r_0}.

\begin{definition}\label{defeuimp}
Let $\delta>0$, $K\in\N$, and $\tau>0$.  We say that $\{x^\tau_{k}\}$ with $k=0,1,\ldots,K$,  is a solution of the discrete gradient flow of $W_{\Lambda,\bf d_0}$ from $x_0$ if $x_0^\tau=x_0$ and, for any $k=1,\ldots,K$,  $x_k^\tau\in\Om^M$ satisfies 
\begin{equation}\label{xktau}
\begin{split}
&x^\tau_k\in\mathrm{argmin}\left\{W_{\Lambda,{d_0}}(x)+\sum_{i=1}^M\frac{\phi^2(x_i-x^{\tau}_{i,k-1})}{2\tau} : \ x\in\Omega^M, \right.\\
&\left.\quad\phantom{x^\tau_k\in\mathrm{argmin} W_{\Lambda,{d_0}}(x)+\sum_{i=1}^M\frac{\phi^2(x_i-x^{\tau}_{i,k-1})}{2\tau}}\, \sum_{i=1}^M |x_i-x_{i,k-1}^\tau|\le\delta\right\} .
\end{split} 
\end{equation}
\end{definition}
With a small abuse of notation we will consider also the case $K=+\infty$, which corresponds to discrete gradient flows defined for all $k\in\N$.
\begin{definition}\label{massimale}
We say that a solution of the discrete gradient flow $\{x_k^\tau\}$ of $W_{\Lambda,\bf d_0}$ from $x_0$  is maximal if the minimum problem in \eqref{xktau} does not admit a solution for $k=K+1$ (or if it is defined for all $k\in \N$).
\end{definition}
For any $\rho>0$, we set
\begin{equation*} 
K_\rho := \left\{x\in\Omega^M\,:\ \min\{\di_{i\ne j}(x_i,x_j),\di(x_i,\partial\Omega)\}\ge{\rho}\right\};
\end{equation*}
By its very definition $W_{\Lambda,\bf d_0}$ is smooth on $K_\rho$ and therefore 
\begin{equation}\label{stimaren}
\max_{x\in K_\rho}|\nabla W_{\Lambda,\bf d_0}(x)| =: M_\rho <+\infty.
\end{equation}
Let $\{x_k^\tau\}$ be a maximal solution of the discrete gradient flow of $W_{\Lambda,\bf d_0}$ from $x_0$, according to Definitions \ref{defeuimp} and \ref{massimale}; for any $ 2 \delta \le r < r_0$, we set
\begin{multline}\label{ktau0}
k_{r}^\tau=k_{r}^\tau(\{x_k^\tau\}):=\inf \{k\in \{1,\ldots, K\}\ :
\\
\ \min\{\di_{i\ne j}(x_{i,k}^{\tau},x_{j,k}^{\tau}),\di(x^{\tau}_{i,k},\partial\Omega)\}\le r\}.
\end{multline}
If $K$ is finite,  since $\sum_{i=1}^M|x_{i,k_{2 \delta}^\tau}^{\tau}-x_{i,k_{2 \delta}^\tau-1}^{\tau}|\le\delta$, 
then  $k_{2 \delta}^\tau < K$; in particular, $k^\tau_r \le k^\tau_{2\delta} < K$. Moreover,  $x_{k}^\tau\in K_\delta$ for any $k=0,1,\ldots,k_{2 \delta}^\tau $.


\begin{proposition}\label{moltoimp}
For $\tau$ small enough and for every $k=1,\ldots,k_{2 \delta}^\tau$, we have that $|x_k^\tau-x_{k-1}^\tau|<\delta$ and 
\begin{equation}\label{soldiscgf}
\frac{x_{i,k}^{\tau}-x_{i,k-1}^{\tau}}{\tau}\in  \partial^- \Psi\left(-\nabla_{x_i} W_{\Lambda,{d_0}}(x_k^\tau)\right)\quad {\text{ for every }}i=1,\ldots,M,
\end{equation}
where $\Psi$ is defined in \eqref{Phi}.
\end{proposition}
\begin{proof}
Since the energy $W_{\Lambda,{d_0}}$ is clearly decreasing in $k$, for every $k=1,\ldots, k_{2\delta}^\tau$ we have
\begin{multline*}
\sum_{i=1}^M\frac{\phi^2(x_{i,k}^{\tau}-x_{i,k-1}^{\tau})}{2\tau}\leq W_{\Lambda,{d_0}}(x_{k-1}^\tau)-W_{\Lambda,{d_0}}(x_k^\tau)\\
\leq W_{\Lambda,{d_0}}(x_0)-W_{\Lambda,{d_0}}(x_k^\tau)\le \max_{x\in K_\delta} (W_{\Lambda,{d_0}}(x_0)-W_{\Lambda,{d_0}}(x)).
\end{multline*}
It follows that for $\tau$ small enough $\sum_{i=1}^M |x_{i,k}^\tau-x_{i,k-1}^\tau|<\delta$.

We notice that the function $\phi^2$ is convex and that $W_{\Lambda,{d_0}}\in C^1(K_\delta)$.
Then by the minimality property of $x_k^\tau$, and by the fact that it belongs to the interior of the constraint  $\{ x: \sum_{i=1}^M |x_{i}-x_{i,k-1}^\tau| \le \delta\}$, we have that for $i=1,\ldots,M$
$$
0\in\nabla_{x_i} W_{\Lambda,{d_0}}(x)|_{x_k^\tau} + \partial^-\left(\frac{\phi^2(x-x_{i,k-1}^{\tau})}{2\tau}\right)_{\big|x_{i,k}^{\tau}}
$$
or equivalently, using the homogeneity of $\phi$,
\begin{equation}\label{LE}
- \nabla_{x_i} W_{\Lambda,{d_0}}(x_k^\tau)\in \partial^-\phi^2\left(\frac{x_{i,k}^{\tau}-x_{i,k-1}^{\tau}}{2\tau}\right).
\end{equation}
Then, by applying the following standard result  (see e.g. \cite[Corollary 5.2]{ET})
\begin{equation}\label{standard}
\xi\in\partial^-\frac{\phi^2}{2}(z)\qquad \Longleftrightarrow \qquad z\in\partial^-\left(\frac{\phi^2}{2}\right)^*(\xi)\qquad\textrm{for any }\xi,z\in\R^2,
\end{equation}
with $\xi=- \nabla_{x_i} W_{\Lambda,{d_0}}(x_k^\tau)$ and $z=\frac{x_{i,k}^{\tau}-x_{i,k-1}^{\tau}}{2\tau}$,
 \eqref{soldiscgf} follows  by \eqref{LE}. 
\end{proof}
In order to study the limit of \eqref{soldiscgf} as $\tau\to 0$, we  recall some classical results for differential inclusions  in \cite{Fil}.
Let $N\in\N$ and let $A$ be an open subset of $\R^N$. 
Let $F$ be  an upper semicontinuous set-valued function  on  $A$ such that for any $x\in A$ the set $F(x)$ is non-empty, bounded and closed.
\begin{definition}\label{tausol}
We say that an absolutely continuous map $x:[0,T]\to A$ is a solution of the differential inclusion
\begin{equation}\label{DifferInc}
\dot{x}(t)\in F(x(t)),
\end{equation}
if it satisfies the inclusion above for almost every $t\in [0,T]$.

Given $\tau>0$, we say that an absolutely continuous map $y:[0,T]\to A$ is a $\tau$-solution (an approximate solution with accuracy $\tau$) of the differential  inclusion \eqref{DifferInc} if 
$$
\dot{y}(t)\in \bigcup_{s\in [t-\tau, t+\tau] } \{ \xi\in \R^N : \di(\xi, \co (F(y(s))) \le \tau\}
$$
{ for almost every $t\in [0,T]$.}
\end{definition}
\begin{proposition}\label{esistenza}
Let $x{ :[0,T]\to A}$ be { the  limit as $\tau\to 0$ of a uniform convergent sequence of $\tau$-solutions $x^\tau:[0,T]\to A$} of the { differential} inclusion \eqref{DifferInc}. Then, $x(t)$ solves \eqref{DifferInc} with $F(x(t))$ replaced by $co(F(x(t)))$, i.e., for a.e. $t\in[0,T]$ 
\begin{equation*} 
\dot{x}(t)\in \co(F(x(t))).
\end{equation*}
\end{proposition}

Now for $0\le t\le k_{2 \delta}^\tau\tau$, we denote by $x^\tau(t)=(x_1^{\tau}(t),\ldots, x_M^{\tau}(t))$ the piecewise affine interpolation in time of the discrete gradient flow $\{x_k^\tau\}$ defined in Definitions~\ref{defeuimp} and~\ref{massimale}.
The next theorem clarifies the limit problem solved by $\{x^\tau\}$ as $\tau\to 0$.
\begin{theorem}\label{thmtautozero}
Let $0<r<r_0$ and let   $0<\delta< \frac  r 2$,
where $r_0$ is defined in \eqref{r_0}.
Let $\{x_k^\tau\}$ be a family of maximal solutions of the discrete gradient flow of $W_{\Lambda, \bf d_0}$ from $x_0$. Then, 
$$
 T_{r}:=\liminf_{\tau\to 0} k_{r}^\tau \tau >0.
$$ 
Moreover,  up to a subsequence $x^\tau\to x$ locally uniformly in $[0, T_{r})$, where  $x:[0, T_{r})\to\Omega^M$ is a Lipschitz continuous map that solves 
\begin{equation}\label{diffinc}
\left\{\begin{array}{l}
\!\!\!\dot{x}_i(t)\!\in \!\partial^{-}\Psi\left(-\nabla_{x_i} \!W_{\Lambda,{d_0}}(x(t))\right)\,\text{for }i=1,\ldots,M,\,\textrm{for a.e. } t\in[0, T_{r})\\
\!\!\!x(0)=x_0,
\end{array}\right.
\end{equation}
and satisfies
 \begin{equation}\label{mindistlim}
 \inf\{\min\{\di_{i\neq j}(x_{i}( t),x_{j}( t)),\di(x_{i}( t),\partial\Omega)\}: t<T_r\}
\geq r.
\end{equation}
If $T_r<+\infty$, then {
\begin{equation}\label{mindistlim2}
\min\{\di_{i\neq j}(x_{i}( T_{r}),x_{j}( T_{r})),\di(x_{i}( T_{r}),\partial\Omega)\}=r.
\end{equation}}
\end{theorem}
\begin{proof}
By the very definition of $k_{r}^\tau$, it is easy to prove that, whenever it is finite, then  
\begin{equation}\label{rb}
|x^\tau_{k_r^\tau}-x_0^\tau|\ge\max_{i=1,\ldots,M} |x_{i,k_r^{\tau}}^{\tau}-x_{i,0}^{\tau}|>\frac{r_0-r}{2}\,.
\end{equation}
Since $2 \delta<r$, $k_r^\tau\le k_{2 \delta}^\tau$; hence, by \eqref{soldiscgf}, for every $k=1,\ldots,k_r^\tau$
$$
 | x_{k}^{\tau}-x_{k-1}^{\tau}| \le C_\phi \max_{x\in K_\delta}|\nabla W_{\Lambda,{d_0}}(x)| \tau = C_\phi M_\delta \tau,
$$
where $C_\phi $ depends only on $\phi$. Therefore,
$$
|x_{k_r^\tau}^\tau-x_{0}^\tau|\le
\sum_{k=1}^{k_{r}^{\tau}}|x_{k}^{\tau}-x_{k-1}^{\tau}|
\le C_{\phi} M_\delta k_{r}^{\tau}\tau,
$$
and hence, by \eqref{rb},
$$
k_{r}^\tau\tau\ge\frac{r_0-r}{2 \dd  C_{\phi}M_\delta}>0.
$$
We deduce that $T_r>0$. Clearly  $x^\tau$ are equi-bounded and Lipschitz equi-continuous in $[0,k_r^\tau\tau]$. By
 Ascoli-Arzel\`a Theorem and by a standard diagonal argument,  up to a subsequence, $x^\tau$ converges locally uniformly on $[0, T_r)$  to a Lipschitz continuous  function $x: [0, T_r]\to\Omega^M$. 
By \eqref{soldiscgf}, for every $T< T_r$ and for $\tau$ small enough, $x^\tau$  satisfy
$$
\left\{\begin{array}{l}
\dot{x}_i^\tau(t)\in  \partial^{-}\Psi(-\nabla_{x_i} W_{\Lambda,{d_0}}(x^\tau_{\lfloor t/\tau\rfloor})) \qquad \text{for }i=1,\ldots,M, \textrm{ for a.e. } t\in[0,T],\\
x(0)=x_0.
\end{array}\right.
$$
Therefore, $x^\tau$ is  a $\tau$-solution to \eqref{diffinc} in $[0,T]$, according with Definition \ref{tausol}.
Moreover, by the very definition of the subdifferential, $\partial^-\Psi(\xi)$ is closed and convex and, since $\Psi$ is convex, it is non-empty, bounded and upper-semicontinuous for every $x\in K_\delta$. 

Therefore, the set-valued function $\partial^-\Psi(-\nabla_{x_i} W_{\Lambda,{d_0}}(x))$ satisfies the assumptions of Proposition \ref{esistenza}, and hence $x(t)$ is a solution to \eqref{diffinc}.
Finally, \eqref{mindistlim} { and \eqref{mindistlim2}} follow by the very definition of $k_r^\tau$.
\end{proof}
\subsection{Asymptotic discrete gradient flow of $\mathcal F_{\ep,\Lambda}$ as $\e\to 0$}\label{taua0}
\qquad

\noindent We are now in a position to state and prove the convergence of discrete gradient flows $\mathcal F_{\ep,\Lambda}$ as $\ep\to 0$.

For any $0<r<r_0$ with $r_0$ defined as in \eqref{r_0}, we set $\delta_r:=\min\{ \frac{r}{2},  \frac{c_0}{M_{\frac{r}{2}}}\}$ with $M_{\frac{r}{2}}$ and $c_0$ defined in \eqref{stimaren} and Lemma \ref{energy-dipole} respectively.

\begin{theorem}\label{MM1}
{ Let $\mu_0\in X(\Om)$ and  fix $0<r<r_0$, $0<\delta<\delta_r$. Let $\{\mu_{\ep,0}\}\subset X_{\ep,\Lambda}(\Om)$ be well prepared initial conditions with respect to $\mu_0$ according with Definition~\ref{wpid}, let $\tau>0$ and let  $\{\mu_{\ep,k}^\tau\}$ be a solution of the discrete gradient flow of $\mathcal{F}_{\ep,\Lambda}$ from $\mu_{\ep,0}$.}
Then, for any $k\in\N$ there exists $\mu_k^\tau\in X(\Omega)$ with $|\mu_k^\tau|(\Om)\le M$, such that, up to subesquences, $\mu_{\ep,k}^\tau\fla\mu_k^\tau$ as $\ep\to 0$. 

Moreover, there exists a maximal solution  $x_k^\tau=(x_{1,k}^\tau,\ldots,x_{M,k}^\tau)$ of the discrete gradient flow of  $W_{\Lambda,{d_0}}$ from $x_0$, according with Definition~\ref{defeuimp}, such that 
$$
\mu_k^\tau=\sum_{i=1}^M d_{i,0}\delta_{x_{i,k}^\tau}\qquad  \hbox{for every $k=1,\ldots, k_r^\tau$,}
$$  
where $k_r^\tau$ is defined in \eqref{ktau0}.
Finally, for any $k=1,\ldots, k_r^\tau$, $\{\mu_{\ep,k}^\tau\}$ is a recovery sequence for $\mu_k^\tau$ in the sense of \eqref{gammalimsup}.
\end{theorem}

\begin{proof}
Since $\mathcal{F}_{\ep,\Lambda}(\mu_{\ep,k}^\tau)$ is not increasing in $k$, we have
\begin{equation}\label{enerbound}
\mathcal{F}_{\ep,\Lambda}(\mu_{\ep,k}^\tau)\le\mathcal{F}_{\ep,\Lambda}(\mu_{\ep,0})\le M\,\lambda_{Q}|\log\ep|+C.
\end{equation}
By the compactness property stated in the $\Gamma$-convergence condition (i), we have that, up to a subsequence, $\mu_{\ep,k}^{\tau}\fla\mu_{k}^{\tau}\in X(\Omega)$,  
with $|\mu_k^\tau|(\Om)\le M$. 

Let now $0<r< {r_0} $  and let  $0<\delta<\delta_r$.  We set
\begin{multline*} 
\tilde k_r^\tau:=\sup\{k\in\N :\ \mu_l^\tau=\sum_{i=1}^M d_{i,0}\delta_{x_{i,l}^\tau},\\
\min\{\di_{i\ne j}(x_{i,l}^\tau,x_{j,l}^\tau),\di(x_{i,l}^\tau,\partial\Omega)\}> r,\ l=0,\ldots,k\}.
\end{multline*}

Since $|\mu_{\tilde k_r^\tau+1}^\tau|(\Om)\le M$ and $\|\mu_{\tilde k_r^\tau+1}^\tau-\mu_{\tilde k_r^\tau}^\tau\|_\flt\le\delta < \frac r 2$, we deduce that, whenever $\tilde k_r^\tau$ is finite,   then
$\mu^\tau_{\tilde k_r^\tau + 1}=\sum_{i=1}^M d_{i,0}\delta_{x_{i, \tilde k_r^\tau + 1}^\tau}$, with
\begin{equation}\label{vincoloviolato}
\frac{r}{2} \le r-\delta\le\min\{ \di_{i\ne j}( x_{i,\tilde k_r^\tau+1}^\tau,x_{j,\tilde k_r^\tau+1}^\tau),\di(x_{i,\tilde k_r^\tau+1}^\tau,\partial\Omega)\}\le  r.
\end{equation}

By induction on $k$, we prove that for any $k=0,1,\ldots,\tilde k_r^\tau+1$, $\{\mu_{\ep,k}^\tau\}$ is a recovery sequence for $\mu_{k}^\tau$ in the sense of \eqref{gammalimsup}. Indeed, by assumption, the claim is satisfied for $k=0$. Assuming that the claim holds true for $k-1$, we prove it for $k$. 
Indeed by \eqref{enerbound} and the definition of $\tilde k_r^\tau$, in view of Proposition \ref{corimp}, there exists a positive atomic measure $\nu$ with integer weights such that
\begin{equation}
|\mu_{\ep,k}^\tau|\res({\R^2\setminus\supp\mu_k^\tau})\weakstar \nu\quad\mbox{in } \R^2\setminus\supp\mu_k^\tau,
\end{equation}
and
\begin{equation}\label{equ1}
\liminf_{\ep\to 0}\F_{\ep,\Lambda}(\mu_{\ep,k}^\tau)-M\frac{\lambda_{Q}}{2\pi}|\log\ep|\ge \W_{Q}(\mu_k^\tau)+M\gamma +c_0\,\nu(\R^2\setminus\supp\mu_k^\tau),
\end{equation}
where $c_0$ is given by Lemma~\ref{energy-dipole}.
Using that $\{\mu_{\ep,k}^\tau\}$ satisfies \eqref{MMeps} and the inductive assumption, we get
\begin{multline}\label{equ2}
\limsup_{\ep\to 0}\F_{\ep,\Lambda}(\mu_{\ep,k}^\tau)-M\frac{\lambda_{Q}}{2\pi}|\log\ep|\\
\le  \limsup_{\ep\to 0}\F_{\ep,\Lambda}(\mu_{\ep,k-1}^\tau)-M\frac{\lambda_{Q}}{2\pi}|\log\ep|
=\W_{Q}(\mu^\tau_{k-1})+M\,\gamma.
\end{multline}
By \eqref{equ1}, \eqref{equ2} and \eqref{vincoloviolato} we have
$$
c_0\,\nu(\R^2\setminus\supp\mu_k^\tau)\le\W_Q(\mu_{k-1}^\tau)-\W_Q(\mu_k^\tau)=W_{\Lambda,{d_0}}(x_{k-1}^\tau)-W_{\Lambda,{d_0}}(x_{k}^\tau) \le M_{\frac{r}{2}} \delta,
$$
where $M_{\frac{r}{2}}$ is defined in \eqref{stimaren}. 
By assumption $M_{\frac{r}{2}} \delta <c_0$, whence $\nu=0$.

Let now $\{\tilde \mu_{\ep,k}^\tau\}$ be a recovery sequence for $\mu_k^\tau$;
since  $\|\mu_k^{\tau}-\mu_{k-1}^\tau\|_{\flt}\le\delta$, by  standard density arguments  we can assume that also  $\|\tilde\mu_{\ep,k}^\tau-\mu_{\ep,k-1}^\tau\|_{\flt}\le\delta$.
Then, by \eqref{equ1} and using again that $\{\mu_{\ep,k}^\tau\}$ satisfies \eqref{MMeps}, we obtain
\begin{multline*} 
\W_Q(\mu_k^\tau)+M\,\gamma\le \liminf_{\ep\to 0}\F_{\ep,\Lambda}(\mu_{\ep,k}^\tau)-M\frac{\lambda_{Q}}{2\pi}|\log\ep|\\
\le
\limsup_{\ep\to 0}\F_{\ep,\Lambda}(\mu_{\ep,k}^\tau)-M\frac{\lambda_{Q}}{2\pi}|\log\ep|\\
\le \limsup_{\ep\to 0}\F_{\ep,\Lambda}(\tilde \mu_{\ep,k}^\tau)-M\frac{\lambda_{Q}}{2\pi}|\log\ep|+
\frac{\diss (\tilde\mu_{\ep,k}^\tau,\mu_{\ep,k-1}^\tau)}{2\tau}-\frac{\diss (\mu_{\ep,k}^\tau,\mu_{\ep,k-1}^\tau)}{2\tau}\\
=\W_Q(\mu_k^\tau)+M\,\gamma\,,
\end{multline*}
where in the last equality we have used that 
$$
\lim_{\ep\to 0}\diss (\tilde\mu_{\ep,k}^\tau,\mu_{\ep,k-1}^\tau)=\lim_{\ep\to 0}\diss (\mu_{\ep,k}^\tau,\mu_{\ep,k-1}^\tau)=\diss(\mu_k^\tau,\mu_{k-1}^\tau),
$$
which holds true since the sequences $\{\mu_{\ep,k-1}^\tau\},\,\{\mu_{\ep,k}^\tau\},\,\{\tilde \mu_{\ep,k}^\tau\},$ satisfy the assumptions of Lemma \ref{lemmaimp} (see also Corollary \ref{contrecov}). 
This concludes the proof that $\{ \mu_{\ep,k}^\tau\}$ is a recovery sequence for $\mu_k^\tau$.

It remains to show that $\{x_k^\tau\}$ is a solution to the discrete gradient flow of $W_{\Lambda,{d_0}}$ from $x_0$ for $k=1,\ldots,\tilde k_r^\tau+1$, and that $k_r^\tau=\tilde k_r^\tau+1$.
Let $x=(x_1,\ldots,x_M)\in\Omega^M$ be such that $\sum_{i=1}^M|x_i-x_{i,k-1}^\tau|\le\delta$. Set $\mu:=\sum_{i=1}^M d_{i,0}\delta_{x_i}$ and let $\{\mu_{\ep}\}$ be a recovery sequence of $\mu$. As above we may assume that $\|\mu_\ep-\mu_{\ep,k-1}^\tau\|_\flt\le\delta$. 

Notice that Corollary \ref{contrecov} implies
\begin{equation}\label{equ3}
\lim_{\ep\to 0}\diss(\mu_{\ep,k}^\tau,\mu_{\ep,k-1}^\tau)=\diss(\mu^\tau_k,\mu_{k-1}^\tau),\quad\lim_{\ep\to 0}\diss(\mu_{\ep},\mu_{\ep,k-1}^\tau)=\diss(\mu,\mu_{k-1}^\tau).
\end{equation}
Therefore, since $\mu_{\ep,k}^\tau$ is a recovery sequence for $\mu_k^\tau$ and by \eqref{MMeps} and \eqref{equ3}, we get
\begin{multline*}
W_{\Lambda,{d_0}}(x_k^\tau)+M\gamma+\sum_{i=1}^M\frac{\phi^2(x_{i,k}^\tau-x_{i,k-1}^\tau)}{2\tau}
= \W_{Q}(\mu_{k}^\tau)+M\gamma+\frac{ \diss(\mu_{k}^\tau,\mu_{k-1}^\tau)}{2\tau}
\\
= \lim_{\ep\to 0}\mathcal{F}_{\ep}(\mu_{\ep,k}^\tau)- M \lambda_Q |\log\ep|+
\frac{\diss(\mu_{\ep,k}^\tau,\mu_{\ep,k-1}^\tau)}{2\tau}
\\
\le\lim_{\ep\to 0}\mathcal{F}_{\ep}(\mu_{\ep})- M \lambda_Q|\log\ep|+\frac{\diss(\mu_{\ep},\mu_{\ep,k-1}^\tau)}{2\tau}\\
=\W_{Q}(\mu)+M\gamma+\frac{\diss(\mu,\mu_{k-1}^\tau)}{2\tau}=W_{\Lambda,{d_0}}(x)+M\gamma+\sum_{i=1}^M\frac{\phi^2(x_{i}-x_{i,k-1}^\tau)}{2\tau},
\end{multline*} 
i.e., $x_k^\tau$ satisfies \eqref{xktau}. 

Finally, by \eqref{vincoloviolato}, we immediately  get that $k_r^\tau=\tilde k_r^\tau+1$. 
\end{proof}

\begin{remark}
\rm{  
We remark that if $\phi$ is the Euclidean norm and the lattice $\Lambda$ is the square lattice $\Z^2$, Theorem \ref{MM1} corrects the statement \cite[Theorem 6.7]{ADGP1}, where the needed assumption that the initial conditions are well prepared (see Definition~\ref{wpid}) was missing. 
The incorrect proof of \cite[Theorem 6.7]{ADGP1} was based on the  claim that the dissipation is always continuous with respect to the flat convergence. Such a claim is in general  wrong, as clarified in Example~\ref{example}.
Nevertheless, in the proof of Theorem \ref{MM1} we have shown that the solutions of the discrete gradient flow  from well prepared initial conditions remain well prepared  at each time step.   This fact, in view of  Lemma~\ref{lemmaimp},  implies the desired continuity property of the dissipation as $\e\to 0$ and allows to conclude the proof of the Theorem \ref{MM1} and, in turn, of \cite[Theorem 6.7]{ADGP1}.

}
\end{remark}

{\bf Acknowledgements:}
RA, AG and MP are members of  the Gruppo Nazionale per l'Analisi Matematica, la Probabilit\`a e le loro Applicazioni (GNAMPA) of the Istituto Nazionale di Alta Matematica (INdAM). The research of LDL is funded by the DFG Collaborative Research Center CRC 109 ``Discretization in Geometry and Dynamics''.

\end{document}